\documentclass[12pt]{amsart}
\usepackage{amsmath, amsthm, amscd, amsfonts, amssymb, graphicx, color,float,pgf,tikz}
\usepackage{amssymb,fontenc}
\usepackage{latexsym,wasysym,mathrsfs}
\usepackage{hyperref}
\usetikzlibrary{arrows}
\usepackage[square,numbers,sort&compress]{natbib}

\makeatletter \oddsidemargin.9375in \evensidemargin \oddsidemargin
\newtheorem{theorem}{Theorem}[section]
\newtheorem{lemma}[theorem]{Lemma}
\newtheorem{proposition}[theorem]{Proposition}
\newtheorem{corollary}[theorem]{Corollary}

\theoremstyle{definition}
\newtheorem{definition}[theorem]{Definition}
\newtheorem{example}[theorem]{Example}

\newtheorem{remark}[theorem]{Remark}
\numberwithin{equation}{section}

\newcommand{\be}{\begin{equation}}
\newcommand{\ee}{\end{equation}}
\newcommand{\bee}{\begin{example}}
\newcommand{\eee}{\end{example}}

\numberwithin{equation}{section}

\usepackage{etoolbox}
\makeatletter
\patchcmd{\@settitle}{\uppercasenonmath\@title}{}{}{}
\patchcmd{\@setauthors}{\MakeUppercase}{}{}{}
\makeatother

\allowdisplaybreaks
\begin{document}

\title[ Continuous Controlled K-G-Frames for Hilbert  $C^\ast$-modules ]{Continuous Controlled K-G-Frames for Hilbert  $C^\ast$-module}

\author[Abdeslam TOURI, Hatim LABRIGUI and , Samir KABBAJ]{Abdeslam TOURI$^1$$^{*}$, Hatim LABRIGUI$^1$ \MakeLowercase{and} Samir KABBAJ$^1$}

\address{$^{1}$Department of Mathematics, University of Ibn Tofail, B.P. 133, Kenitra, Morocco}
\email{\textcolor[rgb]{0.00,0.00,0.84}{touri.abdo68@gmail.com}}

\email{\textcolor[rgb]{0.00,0.00,0.84}{hlabrigui75@gmail.com}}

\email{\textcolor[rgb]{0.00,0.00,0.84}{samkabbaj@yahoo.fr}}

\subjclass[2010]{42C15,42C40}

\keywords{Continuous K-g-frame,Controlled continuous g-frames, Controlled K-g-frame, Continuous Controlled K-g-frame,   $C^{\ast}$-algebra, Hilbert $\mathcal{A}$-modules.\\
\indent $^{*}$ Corresponding author}
\maketitle

\begin{abstract}
	Frame Theory has a great revolution for recent years. This Theory has been extended from Hilbert spaces to Hilbert  $C^{\ast}$-modules. The purpose of this paper is the introduction and the study of the new concept that of Continuous Controlled K-g-Frame for Hilbert $C^{\ast}$-Modules wich is a generalizations of discrete Controlled K-g-Frames in Hilbert $C^{\ast}$-Modules. Also  we establish some results.

\end{abstract}
\maketitle
\vspace{0.1in}

\section{\textbf{Introduction and preliminaries}}
The concept of frames in Hilbert spaces has been introduced by
Duffin and Schaeffer \cite{Duf} in 1952 to study some deep problems in nonharmonic Fourier series. After the fundamental paper \cite{13} by Daubechies, Grossman and Meyer, frame theory began to be widely used, particularly in the more specialized context of wavelet frame and Gabor frame \cite{Gab}. Frames have been used in signal processing, image processing, data compression and sampling theory. 
The concept of a generalization of frame to a family indexed by some locally compact space endowed with a Radon measure was proposed by G. Kaiser \cite{15} and independently by Ali, Antoine and Gazeau \cite{11}. These frames are known as continuous frames. Gabardo and Han in \cite{14} called them frames associated with measurable spaces, Askari-Hemmat, Dehghan and Radjabalipour in \cite{12} called them generalized frames and in mathematical physics they are know as energy-staes. 

In 2012, L. Gavruta \cite{02} introduced the notion of K-frame in Hilbert space to study the atomic systems with respect to a bounded linear operator K.
Controlled frames in Hilbert spaces have been introduced by P. Balazs \cite{01} to improve the numerical efficiency of iterative algorithms for inverting the frame operator. Rahimi \cite{05} defined the concept of controlled K-frames in Hilbert spaces and showed that controlled K-frames are equivalent to K-frames for which the controlled operator C  can be used as precondition in applications.
Controlled frames in $C^{\ast}$-modules were introduced by Rashidi and Rahimi \cite{03}, where the authors showed that they share many useful properties with their corresponding notions in a Hilbert spaces. Finally, we note that Controlled K-g- frames in Hilbet spaces have been introduced by Dingli Hua and Yongdong Huang \cite{F11}. %The continuous controlled K-g- frame in Hilbert spaces were introduced by \\
%Controlled frames in $C^{\ast}$-modules were introduced by Rashidi and Rahimi \cite{03}, and the authors showed that they share many useful properties with their corresponding notions in a Hilbert space. We extended the results of frames in Hilbert spaces to  Hilbert $C^{\ast}$-modules ( \cite{R1},  \cite{Ross},  \cite{R3}, \cite{R4}, \cite{R5}, \cite{R6}, \cite{R7}, \cite{R8}, \cite{R9}, \cite{R10}, \cite{R11}, \cite{R12}).\\
In this paper we introduce the notion of a continuous controlled K-g-frame in Hilbert $C^{\ast}$-modules.

In the following we briefly recall the definitions and basic properties of $C^{\ast}$-algebras and Hilbert $\mathcal{A}$-modules. Our references for $C^{\ast}$-algebras are \cite{{Dav},{Con}}. For a $C^{\ast}$-algebra $\mathcal{A}$ if $a\in\mathcal{A}$ is positive we write $a\geq 0$ and $\mathcal{A}^{+}$ denotes the set of all positive elements of $\mathcal{A}$.
\begin{definition}\cite{Pas}	
	Let $ \mathcal{A} $ be a unital $C^{\ast}$-algebra and $\mathcal{H}$ be a left $ \mathcal{A} $-module, such that the linear structures of $\mathcal{A}$ and $ \mathcal{H} $ are compatible. $\mathcal{H}$ is a pre-Hilbert $\mathcal{A}$-module if $\mathcal{H}$ is equipped with an $\mathcal{A}$-valued inner product $\langle.,.\rangle_{\mathcal{A}} :\mathcal{H}\times\mathcal{H}\rightarrow\mathcal{A}$, such that is sesquilinear, positive definite and respects the module action. In the other words,
	\begin{itemize}
		\item [(i)] $ \langle x,x\rangle_{\mathcal{A}}\geq0 $ for all $ x\in\mathcal{H} $ and $ \langle x,x\rangle_{\mathcal{A}}=0$ if and only if $x=0$.
		\item [(ii)] $\langle ax+y,z\rangle_{\mathcal{A}}=a\langle x,z\rangle_{\mathcal{A}}+\langle y,z\rangle_{\mathcal{A}}$ for all $a\in\mathcal{A}$ and $x,y,z\in\mathcal{H}$.
		\item[(iii)] $ \langle x,y\rangle_{\mathcal{A}}=\langle y,x\rangle_{\mathcal{A}}^{\ast} $ for all $x,y\in\mathcal{H}$.
	\end{itemize}	 
	For $x\in\mathcal{H}, $ we define $||x||=||\langle x,x\rangle_{\mathcal{A}}||^{\frac{1}{2}}$. If $\mathcal{H}$ is complete with $||.||$, it is called a Hilbert $\mathcal{A}$-module or a Hilbert $C^{\ast}$-module over $\mathcal{A}$. For every $a$ in $C^{\ast}$-algebra $\mathcal{A}$, we have $|a|=(a^{\ast}a)^{\frac{1}{2}}$ and the $\mathcal{A}$-valued norm on $\mathcal{H}$ is defined by $|x|=\langle x, x\rangle_{\mathcal{A}}^{\frac{1}{2}}$ for $x\in\mathcal{H}$.
	
	Let $\mathcal{H}$ and $\mathcal{K}$ be two Hilbert $\mathcal{A}$-modules. A map $T:\mathcal{H}\rightarrow\mathcal{K}$ is said to be adjointable if there exists a map $T^{\ast}:\mathcal{K}\rightarrow\mathcal{H}$ such that $\langle Tx,y\rangle_{\mathcal{A}}=\langle x,T^{\ast}y\rangle_{\mathcal{A}}$ for all $x\in\mathcal{H}$ and $y\in\mathcal{K}$.
	
We reserve the notation $End_{\mathcal{A}}^{\ast}(\mathcal{H},\mathcal{K})$ for the set of all adjointable operators from $\mathcal{H}$ to $\mathcal{K}$ and $End_{\mathcal{A}}^{\ast}(\mathcal{H},\mathcal{H})$ is abbreviated to $End_{\mathcal{A}}^{\ast}(\mathcal{H})$.
\end{definition}

The following lemmas will be used to prove our mains results

\begin{lemma} \label{1} \cite{Ara}.
	Let $\mathcal{H}$ and $\mathcal{K}$ two Hilbert $\mathcal{A}$-modules and $T\in End_{\mathcal{A}}^{\ast}(\mathcal{H},\mathcal{K})$. Then the following statements are equivalente,
	\begin{itemize}
		\item [(i)] $T$ is surjective.
		\item [(ii)] $T^{\ast}$ is bounded below with respect to norm, i.e, there is $m>0$ such that $\|T^{\ast}x\|\geq m\|x\|$, $x\in\mathcal{K}$.
		\item [(iii)] $T^{\ast}$ is bounded below with respect to the inner product, i.e, there is $m'>0$ such that, $$\langle T^{\ast}x,T^{\ast}x\rangle_\mathcal{A}\geq m'\langle x,x\rangle_\mathcal{A} , x\in\mathcal{K}$$.
	\end{itemize}
\end{lemma}

\begin{lemma} \label{2}\cite{Pas} 
	Let $\mathcal{H}$ and $\mathcal{K}$ two Hilbert $\mathcal{A}$-modules and $T\in End_{\mathcal{A}}^{\ast}(\mathcal{H},\mathcal{K})$. Then the following statements are equivalente,
	\begin{itemize}
		\item [(i)] The operator T is bounded and $\mathcal{A}$-linear.
		\item  [(ii)]	There exist $0\leq k$ such that $\langle Tx, Tx \rangle_\mathcal{A}\leq  k \langle x, x \rangle_\mathcal{A}  $  for all $x\in \mathcal{H} $.
	\end{itemize}	
\end{lemma} 
%\begin{lemma} \label{do6}
%	Let $ T\in l(\mathcal{H})$.
%	Then $T>0$ if and only if $T$ is injective and $ 0\leq T$.
%\end{lemma}
%\begin{proof}
%	Let $T$ be injective and $0\leq T$.
%	Let $f\neq 0$ and $\langle Tf,f\rangle =0$\\
%	We have, $$\|T^{\frac{1}{2}}f\|^2 =\langle T^{\frac{1}{2}}f , T^{\frac{1}{2}}f\rangle =\langle Tf,f\rangle =0  .$$
%	Hence , $T^{\frac{1}{2}}f =0$.
%	So , $f=0$.
%	Therfore , $T>0$.
%\end{proof}
%\begin{remark}
%	Let $\mathcal{H}$ , $\mathcal{K}$ be a Hilbert $C^{\ast}$-module and $T\in End_{\mathcal{A}}^{\ast}(\mathcal{H},\mathcal{K})$.
%	Since $T^{\ast}T>0$ and T is injective if and only if $T^{\ast}T$ is injective. By lemma\ref{do6}, $T^{\ast}T>0$ if and only if T is injective .
	
%\end{remark}

For the following theorem, $R(T)$ denote the range of the operattor $T$.
\begin{theorem}\label{do5} \cite{Do} 
	Let E, F and G be Hilbert ${\mathcal{A}}$-modules over a $C^{\ast}$-algebra $\mathcal{A}$. Let $ T\in End_{\mathcal{A}}^{\ast}(E,F) $ and  $  T'\in End_{\mathcal{A}}^{\ast}(G,F) $ with 
    $\overline{({R(T^{\ast})})}$ is orthogonally complemented. Then the following statements are equivalente:
	\begin{itemize}
		\item [(1)] $T'(T')^{\ast} \leq \lambda TT^{\ast}$ for some $\lambda >0$.
		\item  [(2)]	There exist $\mu >0 $ such that $\|(T')^{\ast}x\|\leq \mu \|T^{\ast}x\|$ for all $x\in F$.
		\item  [(3)] There exists $ D\in End_{\mathcal{A}}^{\ast}(G,E)$ such that $T'=TD$,
		that is the equation $TX=T'$ has a solution.
		\item  [(4)]  $R(T')\subseteq R(T)$.
		
	\end{itemize}
	
\end{theorem}

\section{\textbf{Continuous Controlled K-G-Frames for Hilbert  $C^\ast$-modules}}
Let $X$ be a Banach space, $(\Omega,\mu)$ a measure space, and  $f:\Omega\to X$ a measurable function. Integral of the Banach-valued function $f$ has been defined by Bochner and others. Most properties of this integral are similar to those of the integral of real-valued functions. Since every $C^{\ast}$-algebra and Hilbert $C^{\ast}$-module is a Banach space thus we can use this integral and its properties.

Let $\mathcal{H}$ and $\mathcal{K}$ be two Hilbert $C^{\ast}$-modules, $\{\mathcal{K}_{w}: w\in\Omega\}$  is a family of subspaces of $\mathcal{K}$, and $End_{\mathcal{A}}^{\ast}(\mathcal{H},\mathcal{K}_{w})$ is the collection of all adjointable $\mathcal{A}$-linear maps from $\mathcal{H}$ into $\mathcal{K}_{w}$.
We define
\begin{equation*}
l^{2}(\Omega, \{\mathcal{K}_{w}\}_{\omega \in \Omega})=\left\{x=\{x_{w}\}_{w\in\Omega}: x_{w}\in \mathcal{K}_{w}, \left\|\int_{\Omega}|x_{w}|^{2}d\mu(w)\right\|<\infty\right\}.
\end{equation*}
For any $x=\{x_{w}: w\in\Omega\}$ and $y=\{y_{w}: w\in\Omega\}$, if the $\mathcal{A}$-valued inner product is defined by $\langle x,y\rangle=\int_{\Omega}\langle x_{w},y_{w}\rangle_\mathcal{A} d\mu(w)$, the norm is defined by $\|x\|=\|\langle x,x\rangle_\mathcal{A}\|^{\frac{1}{2}}$. The $l^{2}(\Omega, \{\mathcal{K}_{w}\}_{\omega \in \Omega})$ is a Hilbert $C^{\ast}$-module\cite{EC}.\\
Let $\mathcal{A}$ be a $C^{\ast}$-algebra, $l^2({\mathcal{A}})$ is defined by, 
$$l^2({\mathcal{A}})=\{\{a_\omega\}_{w\in\Omega} \subseteq \mathcal{A} :\|\int_{\Omega}a_{\omega} a_{\omega}^{\ast} d\mu({\omega})\| <\infty\}.$$
$l^2({\mathcal{A}})$ is a Hilbert $C^{\ast}$-module (Hilbert $\mathcal{A}-module$) with pointwise operations and the inner product defined by, 
$$\langle \{a_\omega\}_{w\in\Omega} , \{b_\omega\}_{w\in\Omega}\rangle =\int_{\Omega}a_{\omega} b_{\omega}^{\ast} d\mu({\omega}), 
\{a_\omega\}_{w\in\Omega}, \{b_\omega\}_{w\in\Omega} \in l^2({\mathcal{A}}),$$
and, $$\|\{a_\omega\}_{w\in\Omega}\|=(\int_{\Omega}a_{\omega} a_{\omega}^{\ast} d\mu({\omega}))^{\frac{1}{2}}.$$
Let $GL^{+}(\mathcal{H})$ be the set of all positive bounded linear invertible operators on $\mathcal{H}$ with bounded inverse.
\begin{definition}\label{do3}
	Let  $\Lambda= \{\Lambda_{w}\}_{w\in\Omega}$ be a family in $ End_{\mathcal{A}}^{\ast}(\mathcal{H},\mathcal{K}_{w})$ for all $\omega \in \Omega$, and $C,C^{'} \in GL^{+}(\mathcal{H})$. We say that the  family $\Lambda$ is a $(C,C^{'})$-controlled continuous g-frame for Hilbert $C^{\ast}$-module $\mathcal{H}$ with respect to  $\{\mathcal{K}_{w}: w\in\Omega\}$ if it is a continuous g-Bessel family and there is a pair of constants $0<A, B$ such that, for any $f\in \mathcal{H}$,
	\begin{equation}\label{DP}
	A\langle f,f\rangle_{\mathcal{A}} \leq\int_{\Omega}\langle\Lambda_{w}Cf,\Lambda_{w}C^{'}f\rangle_{\mathcal{A}} d\mu(w)\leq B\langle f,f\rangle_{\mathcal{A}} \quad .
	\end{equation}
	$A$ and $B$ are called the $(C,C^{'})$-controlled continuous g-frames bounds.\\
	If $C^{'}=I$ then we call $\Lambda= \{\Lambda_{w}\}_{w\in\Omega}$ a $C$-controlled continuous g-frames for $\mathcal{H}$ with respect to  $\{\mathcal{K}_{w}: w\in\Omega\}$.\\
\end{definition}

\begin{definition} \label{17} 
	Let $\mathcal{H}$ be a Hilbert $\mathcal{A}$-module over a unital $C^{\ast}$-algebra, and 
	$ C,C'\in {GL^{+}(\mathcal{H})}$. 
	A family of adjointable operators $\{\Lambda_\omega\}_{w\in \Omega }\subset End_{\mathcal{A}}^{\ast}(\mathcal{H},\mathcal{K}_{w})$
	is said to be a continuous $(C,C^{'})$-controlled  K-g-frame for Hilbert $C^{\ast}$-module $\mathcal{H}$ with respect to  $\{\mathcal{K}_{w}: w\in\Omega\}$ if there exists two positive elements $A$ and $B$  such that
	\begin{equation} \label{127}
	A\langle K^\ast f,K^\ast f\rangle_\mathcal{A} \leq \int_{\Omega}\langle \Lambda_\omega C f,\Lambda_\omega C'f \rangle_\mathcal{A} d\mu(w) \leq B\langle f,f\rangle_\mathcal{A} , f\in \mathcal{H}.
	\end{equation}

	The elements $A$ and $B$ are called continuous $(C,C^{'})$-controlled  K-g-frame bounds. 
	
	If $A=B$ we call this continuous $(C,C^{'})$-controlled  K-g-frame a continuous   $(C,C^{'})$-controlled tight K-g-frame , and if $A=B=1$ it is called a continuous $(C,C^{'})$-controlled Parseval  K-g-Frame. If only the right-hand inequality of \eqref{127} is satisfied, we call  a continuous $(C,C^{'})$-controlled Bessel K-g-frame with Bessel bound $B$.\\
\end{definition}
\begin{example}
	Let $\mathcal{H}=\left\{ M=\left( 
	\begin{array}{cccc}
	a & b & 0 & 0 \\ 
	0 & c & 0 & d%
	\end{array}%
	\right) \text{ / }a,b,c,d\in 
	\mathbb{C}
	\right\} $, \\
	and $\mathcal{A=}\left\{ \left( 
	\begin{array}{cc}
	a & b \\ 
	c & d%
	\end{array}%
	\right) \text{ / }a,b,c,d\in 
	\mathbb{C}
	\right\} $ \\
	It's clair that $\mathcal{H}$ respectively $\mathcal{A}$ is a Hilbert space respectively a $\mathbb{C}^{\ast}$-algebra. Also it's known that $\mathcal{H}$ is a Hilbert $\mathcal{A}$-module.\\
	Let $C$ and $C^{'}$ be two operators respectively defined as follow,
	
	\begin{align*}
		C: \mathcal{H}   &\longrightarrow    \mathcal{H} \\
		 M&\longrightarrow \alpha M
			\end{align*}
			and
\begin{align*}
C^{'}: \mathcal{H}   &\longrightarrow    \mathcal{H} \\
 M&\longrightarrow \beta M
\end{align*}
	where $\alpha$ and $\beta$ are two reels numbers strictly greater than zero.\\
	It's clair that  $	C,C^{'} \in Gl^{+}(\mathcal{H})$.\\
	Indeed, for each $M \in \mathcal{H} $ one has
	\begin{equation*}
C^{-1}(M)=\alpha^{-1}M \quad and \quad (C^{'})^{-1}(M)=\beta^{-1}M.
	\end{equation*}
Let $\Omega = [0,1]$ endewed with the lebesgue's measure. It's clear that a measure space.
Moreover, for $\omega \in \Omega$, we define the operator $\Lambda _{w} : \mathcal{H}\rightarrow \mathcal{H}$ by,
\[
\Lambda _{w}(M)=w\left( 
\begin{array}{cccc}
0 & b & 0 & 0 \\ 
0 & c & 0 & 0%
\end{array}%
\right),
\]\\
$\Lambda _{w}$ is linear, bounded and selfadjoint.\\
In addition, for $M\in \mathcal{H}$, we have,
\begin{align*}
\int_{\Omega}\langle \Lambda _{w}CM,\Lambda _{w}C^{'}M\rangle_{\mathcal{A}}d\mu(\omega)&=\int_{\Omega}\alpha\beta\left( 
\begin{array}{cccc}
0 & b & 0 & 0 \\ 
0 & c & 0 & 0%
\end{array}%
\right)\left( 
\begin{array}{cccc}
0 & b \\
w\bar{b} & w\bar{c}\\
 0 & 0\\
0 & c %
\end{array}%
\right)d\mu(\omega)\\
&=\int_{\Omega}\alpha\beta\left( 
\begin{array}{cccc}
|b|^{2} & b\bar{c}  \\ 
c\bar{b} & |c|^{2} %
\end{array}%
\right)w^{2}d\mu(\omega).
\end{align*}
It's clear that,
\begin{align*}
\left( 
\begin{array}{cccc}
|b|^{2} & b\bar{c}  \\ 
c\bar{b} & |c|^{2} %
\end{array}%
\right)\leq 
\left( 
\begin{array}{cccc}
|a|^{2} + |b|^{2} & b\bar{c}  \\ 
c\bar{b} & |c|^{2} + |d|^{2}%
\end{array}%
\right)=\|M\|_{\mathcal{A}}^{2}.
\end{align*}
Then we have
\begin{equation*}
\int_{\Omega}\langle \Lambda _{w}CM,\Lambda _{w}C^{'}M\rangle_{\mathcal{A}}d\mu(\omega)\leq \frac{\alpha\beta}{3}\|M\|_{\mathcal{A}}^{2}.
\end{equation*}
Which show that the family $(\Lambda_{\omega})_{\omega\in \Omega}$ is a continuous $(C,C^{'})$-controlled Bessel sequence for $\mathcal{H}$ with $\frac{\alpha\beta}{3}$ as bound.\\
But if $b=c=0$, it's impossible to found a positive scalar $A$ such that
\begin{equation*}
A\|M\|_{\mathcal{A}}^{2}\leq \int_{\Omega}\langle \Lambda _{w}CM,\Lambda _{w}C^{'}M\rangle_{\mathcal{A}}d\mu(\omega)=\left( 
\begin{array}{cccc}
0 & 0  \\ 
0 & 0%
\end{array}%
\right),
\end{equation*}
where 
\begin{equation*}
M=\left( 
\begin{array}{cccc}
	a & 0 & 0 & 0 \\ 
	0 & 0 & 0 & d%
\end{array}%
\right) \qquad and \quad a,b >0.
\end{equation*}
So, $(\Lambda_{\omega})_{\omega\in \Omega}$ is not a continuous $(C,C^{'})$-controlled frame for $\mathcal{H}$.\\
But, if we consider the operator 
\begin{align*}
K:\qquad \mathcal{H} \qquad&\longrightarrow \qquad \mathcal{H}\\
\left( 
\begin{array}{cccc}
a & b & 0 & 0 \\ 
0 & c & 0 & d%
\end{array}%
\right)&\longrightarrow \left( 
\begin{array}{cccc}
0 & b & 0 & 0 \\ 
0 & c & 0 & 0%
\end{array}%
\right).
\end{align*}
Wich's linear, bounded and selfadjoint, we found
\begin{equation*}
\langle K^{\ast}M,K^{\ast}M\rangle = \left( 
\begin{array}{cccc}
|b|^{2} & b\bar{c}  \\ 
c\bar{b} & |c|^{2} %
\end{array}%
\right).
\end{equation*}
Then $(\Lambda_{\omega})_{\omega\in \Omega}$ is a continuous $(C,C^{'})$-controlled K-g-frame for $\mathcal{H}$.
\end{example}
    \begin{remark}
     Every continuous $(C,C')-$controlled  g-frame for $ \mathcal{H}$ is a continuous $(C,C')-$controlled  K-g-frame for $ \mathcal{H}$.
     Indeed, if $\{\Lambda_\omega\}_{w\in \Omega }$ is a continuous $(C,C^{'})$-controlled  g-frame for Hilbert $C^{\ast}$-module $\mathcal{H}$ with respect to  $\{\mathcal{K}_{w}: w\in\Omega\}$, then there exist a constants $A,B>0$ such that , $$A\langle  f, f\rangle_\mathcal{A} \leq \int_{\Omega}\langle \Lambda_\omega C f,\Lambda_\omega C'f \rangle_\mathcal{A} d\mu(w) \leq B\langle f,f\rangle_\mathcal{A} , f\in \mathcal{H}.$$
     But, 	\begin{displaymath}
     \langle K^{\ast}f,K^{\ast}f\rangle_\mathcal{A}\leq\|K\|^{2}\langle f,f\rangle_\mathcal{A},  f\in\mathcal{H}. 
     \end{displaymath}
     So, $$A\|K\|^{-2}\langle K^{\ast}f,K^{\ast}f\rangle_\mathcal{A} \leq \int_{\Omega}\langle \Lambda_\omega C f,\Lambda_\omega C'f \rangle_\mathcal{A} d\mu(w) \leq B\langle f,f\rangle_\mathcal{A} , f\in \mathcal{H}.$$
     Hence,  $\{\Lambda_\omega\}_{w\in \Omega }$ is a continuous $(C,C^{'})$-controlled  K-g-frame for Hilbert $C^{\ast}$-module $\mathcal{H}$ with respect to  $\{\mathcal{K}_{w}: w\in\Omega\}$.

    \end{remark}

	Let $\{\Lambda_\omega\}_{i\in \Omega }$ be a continuous $(C,C')-$controlled Bessel K-g-frame for Hilbert $C^{\ast}$- module $\mathcal{H}$ over $\mathcal{A}$ with respect to $\{\mathcal{K}_{w}: w\in\Omega\}$ with bounds $A$ and $B$.\\
	
	We define the operaror $T_{(C,C^{'})}$ by:
	\begin{equation*}
	T_{(C,C^{'})}:l^{2}(\Omega,\{\mathcal{K}_{w}\}_{w\in\Omega}) \rightarrow \mathcal{H},
	\end{equation*}
	
	such that:
	\begin{equation*}
	T_{(C,C^{'})}(\{y_{w}\}_{w\in\Omega})=\int_{\Omega}(CC^{'})^{\frac {1}{2}}\Lambda^{\ast}_{\omega}y_{\omega}d\mu(w), \qquad   \{y_{w}\}_{w\in\Omega} \in l^{2}(\Omega,\{\mathcal{K}_{w}\}_{w\in\Omega}).
	\end{equation*}
	The bounded linear operator $T_{(C,C^{'})}$ is called the $(C,C^{'})$ synthesis operator of $\Lambda$.\\
	The operator:
	\begin{equation*}
	T^{\ast}_{(C,C^{'})}: \mathcal{H}\rightarrow l^{2}(\Omega,\{\mathcal{K}_{w}\}_{w\in\Omega}),
	\end{equation*}  
	is given by:
	\begin{equation}\label{2..3}
	T^{\ast}_{(C,C^{'})}(x)=\{\Lambda_{\omega}(C^{'}C)^{\frac{1}{2}}x\}_{\omega \in \Omega}, \qquad x\in \mathcal{H},
	\end{equation}
	is called the $(C,C^{'})$ analysis operator for $\Lambda$.\\

	Inded, we have for all $x\in \mathcal{H}$ and $\{y_{w}\}_{w\in\Omega} \in l^{2}(\Omega,\{\mathcal{K}_{w}\}_{w\in\Omega})$
	\begin{align*}
	\langle T_{(C,C^{'})}(\{y_{w}\}_{w\in\Omega}),x\rangle_{\mathcal{A}}&=\langle \int_{\Omega}(CC^{'})^{\frac {1}{2}}\Lambda^{\ast}_{\omega}y_{\omega}d\mu(w),x\rangle_{\mathcal{A}}\\
	&=\int_{\Omega}\langle (CC^{'})^{\frac {1}{2}}\Lambda^{\ast}_{\omega}y_{\omega},x\rangle_{\mathcal{A}}d\mu(w)\\
	&=\int_{\Omega}\langle y_{\omega},\Lambda_{\omega}(CC^{'})^{\frac {1}{2}}x\rangle_{\mathcal{A}}d\mu(w)\\
	&=\langle \{y_{w}\}_{w\in\Omega} , \{\Lambda_{\omega}(C^{'}C)^{\frac{1}{2}}x\}_{\omega \in \Omega}\rangle_{l^{2}(\Omega,\{\mathcal{K}_{w}\}_{w\in\Omega})}\\
	&=\langle \{y_{w}\}_{w\in\Omega} , T^{\ast}_{(C,C^{'})}(x)\rangle_{l^{2}(\Omega,\{\mathcal{K}_{w}\}_{w\in\Omega})}.
	\end{align*}
	Which shows that $T^{\ast}_{(C,C^{'})}$ is the adjoint of $T_{(C,C^{'})}$.
	If $C$ and $C^{'}$ commute between them, and commute with the operators $\Lambda^{\ast}_{\omega}\Lambda_{\omega}$ for each $\omega \in \Omega$. We define the $(C,C^{'})$-controlled continuous g-frames operator by:
	\begin{align*}
	S_{(C,C^{'})}:&\mathcal{H}\longrightarrow \mathcal{H}\\
	&x\longrightarrow S_{(C,C^{'})}x=T_{(C,C^{'})}T^{\ast}_{(C,C^{'})}x=\int_{\Omega}C^{'}\Lambda^{\ast}_{w}\Lambda_{w}Cx d\mu(w).
	\end{align*}
	As consequence on has the following proposition.
\begin{proposition} \label{do}
    The operator $S_{(C,C^{'})}$ is positive, sefladjoint, and bounded.
\end{proposition}
\begin{proposition}
	Let $K\in End_{\mathcal{A}}^{\ast}(\mathcal{H})$ and $ C,C'\in {GL^{+}(\mathcal{H})}$. Suppose that C and C' commutes with each other  and commute with the operators $\Lambda^{\ast}_{\omega}\Lambda_{\omega}$ for each $\omega \in \Omega$.. A family $\{\Lambda_\omega\}_{w\in \Omega }$ is a continuous $(C,C^{'})$-controlled  Bessel K-g-frames for $\mathcal{H}$ with respect to $\{\mathcal{K}_{w}: w\in\Omega\}$ with bounds  $B$ if and only if the operator $T_{(C,C^{'})}$ is well defined and bounded with $\|T_{(C,C^{'})}\|\leq \sqrt{B}$.
\end{proposition}
\begin{proof}
   	$(1)\Longrightarrow (2)$\\
   	Let  $\{\Lambda_{w}, w\in\Omega\}$ be a $(C,C^{'})$-controlled continuous K-g-Bessel family for $\mathcal{H}$ with respect $\{\mathcal{K}_{\omega}\}_{\omega \in \Omega}$ with bound $B$.\\
   	Then we have, 
 \begin{equation} \label{2.12}
    \|\int_{\Omega}\langle\Lambda_{w}Cx,\Lambda_{w}C^{'}x\rangle_{\mathcal{A}} d\mu(w)\| \leq B\|x\|^{2},  \qquad  x\in \mathcal{H}.
 \end{equation}
     Indeed for all $\{y_{w}\}_{w\in\Omega} \in l^{2}(\Omega,\{\mathcal{K}_{w}\}_{w\in\Omega})$ :
\begin{align*}
   	\|T_{CC^{'}}(\{y_{w}\}_{w\in\Omega})\|^{2}&=\underset{x \in \mathcal{H}, \|x\|=1}{\sup}\|\langle T_{CC^{'}}(\{y_{w}\}_{w\in\Omega}),x\rangle_{\mathcal{A}} \|^{2}.
\end{align*}
     Hence, 
 \begin{align*}
   	\|T_{CC^{'}}(\{y_{w}\}_{w\in\Omega})\|^{2}&=\underset{x \in \mathcal{H}, \|x\|=1}{\sup}\|\langle \int_{\Omega}(CC^{'})^{\frac {1}{2}}\Lambda^{\ast}_{\omega}y_{\omega}d\mu(w),x\rangle_{\mathcal{A}}\|^{2}\\
   	&=\underset{x \in \mathcal{H}, \|x\|=1}{\sup}\| \int_{\Omega}\langle(CC^{'})^{\frac {1}{2}}\Lambda^{\ast}_{\omega}y_{\omega},x\rangle_{\mathcal{A}} d\mu(w)\|^{2}\\ 
   	&=\underset{x \in U, \|x\|=1}{\sup}\| \int_{\Omega}\langle y_{\omega},\Lambda_{\omega}(CC^{'})^{\frac {1}{2}}x\rangle_{\mathcal{A}} d\mu(w)\|^{2}\\ 
   	&\leq \underset{x \in U, \|x\|=1}{\sup}\| \int_{\Omega}\langle y_{\omega},y_{\omega}\rangle_{\mathcal{A}} d\mu(w)\| \|\int_{\Omega}\langle \Lambda_{\omega}(CC^{'})^{\frac {1}{2}}x,\Lambda_{\omega}(CC^{'})^{\frac {1}{2}}x\rangle_{\mathcal{A}} d\mu(w)\|\\ 
   	&= \underset{x \in U, \|x\|=1}{\sup}\| \int_{\Omega}\langle y_{\omega},y_{\omega}\rangle_{\mathcal{A}} d\mu(w)\| \|\int_{\Omega}\langle \Lambda_{\omega}Cx,\Lambda_{\omega}C^{'}x\rangle_{\mathcal{A}} d\mu(w)\|\\ 
   	&\leq \underset{x \in \mathcal{H}, \|x\|=1}{\sup}\| \int_{\Omega}\langle y_{\omega},y_{\omega}\rangle_{\mathcal{A}} d\mu(w)\|B\|x\|^{2} =B\|\{y_{\omega}\}_{\omega \in \Omega}\|^{2}.
 \end{align*}
    Then we have 
 \begin{align*}
    \|T_{CC^{'}}(\{y_{w}\}_{w\in\Omega})\|^{2}\leq B\|\{y_{\omega}\}_{\omega \in \Omega}\|^{2} \Longrightarrow \|T_{CC^{'}}\|\leq \sqrt{B}. 
 \end{align*}
   	We conclude that the operator $T_{CC^{'}}$ is well defined and bounded.\\
    	$(2)\Longrightarrow (1)$\\
    	If (2) holds, then for any $x\in \mathcal{H}$, we have:
    	
    	\begin{align*}
    	\int_{\Omega}\langle\Lambda_{w}Cx,\Lambda_{w}C^{'}x\rangle_{\mathcal{A}} d\mu(w)&=\int_{\Omega}\langle C^{'}\Lambda^{\ast}_{w}\Lambda_{w}Cx,x\rangle_{\mathcal{A}} d\mu(w)\\
    	&=\int_{\Omega}\langle (CC^{'})^{\frac {1}{2}}\Lambda^{\ast}_{w}\Lambda_{w}(CC^{'})^{\frac {1}{2}}x,x\rangle_{\mathcal{A}} d\mu(w)\\
    	\end{align*}
    	Put: $y_{w}=\Lambda_{w}(CC^{'})^{\frac {1}{2}}x$ , $\omega \in \Omega$, then:\\
    	\begin{align*}
    	\int_{\Omega}\langle\Lambda_{w}Cx,\Lambda_{w}C^{'}x\rangle_{\mathcal{A}} d\mu(w)&=\langle T_{CC^{'}}(\{y_{w}\}_{w\in\Omega}),x\rangle_{\mathcal{A}}\\
    	&\leq \| T_{CC^{'}}\|\|(\{y_{w}\}_{w\in\Omega})\|\|x\|\\
    	&\leq \| T_{CC^{'}}\|(\int_{\Omega}\|\Lambda_{w}(CC^{'})^{\frac {1}{2}}x\|^{2}d\mu(w))^{\frac {1}{2}}\|x\|\\
    	&\leq\| T_{CC^{'}}\|\|x\|(\int_{\Omega}\langle\Lambda_{w}Cx,\Lambda_{w}C^{'}x\rangle_{\mathcal{A}} d\mu(w))^{\frac {1}{2}}.
    	\end{align*}
    	So, 
    	\begin{equation*}
    	(\int_{\Omega}\langle\Lambda_{w}Cx,\Lambda_{w}C^{'}x\rangle_{\mathcal{A}} d\mu(w))^{\frac{1}{2}}\leq \| T_{CC^{'}}\|\|x\|.
    	\end{equation*}
    	Therefore:
    	\begin{align*}
    	\int_{\Omega}\langle\Lambda_{w}Cx,\Lambda_{w}C^{'}x\rangle_{\mathcal{A}} d\mu(w) &\leq \| T_{CC^{'}}\|^{2}\|x\|^{2}.
    	\end{align*}
    	As  $\| T_{CC^{'}}\|\leq \sqrt{B}$, we have :
    	\begin{align*}
    	\int_{\Omega}\langle\Lambda_{w}Cx,\Lambda_{w}C^{'}x\rangle_{\mathcal{A}} d\mu(w) &\leq B\|x\|^{2}, 
    	\end{align*}
    	which end the proof.
    \end{proof}
    
\begin{lemma}
	 Let $\{\Lambda_\omega\}_{w\in \Omega }\subset End_{\mathcal{A}}^{\ast}(\mathcal{H},\mathcal{K}_{w})$ be a continuous $(C,C')-$controlled Bessel K-g-frame for Hilbert $C^{\ast}$- module $\mathcal{H}$  with respect to $\{\mathcal{K}_{w}: w\in\Omega\}$. Then for any $K \in End_{\mathcal{A}}^{\ast}(\mathcal{H})$, the family $\{\Lambda_\omega K \}_{w\in \Omega }$ is a continuous $(C,C')-$controlled Bessel K-g-frame for Hilbert $C^{\ast}$- module $\mathcal{H}$.\\
\end{lemma}
\begin{proof}
	Assume that  $\{\Lambda_\omega\}_{w\in \Omega }$ is a continuous $(C,C')-$controlled Bessel K-g-frame for Hilbert $C^{\ast}$- module $\mathcal{H}$  with respect to $\{\mathcal{K}_{w}: w\in\Omega\}$ with bound B.
	Then, $$\int_{\Omega}\langle \Lambda_\omega C f,\Lambda_\omega C'f \rangle_\mathcal{A} d\mu(w) \leq B\langle f,f\rangle_\mathcal{A} , f\in \mathcal{H}.$$
	So, $$\int_{\Omega}\langle \Lambda_\omega C Kf,\Lambda_\omega C'Kf \rangle_\mathcal{A} d\mu(w) \leq B\langle Kf,Kf\rangle_\mathcal{A} , f\in \mathcal{H}.$$
	Hence,$$\int_{\Omega}\langle \Lambda_\omega K C f,\Lambda_\omega K C'f \rangle_\mathcal{A} d\mu(w) \leq B\langle Kf,Kf\rangle_\mathcal{A}\leq \|K\|^2 B\langle f,f\rangle_\mathcal{A} , f\in \mathcal{H}.$$
	The results holds.
\end{proof}
\begin{lemma} 
	Let $K \in End_{\mathcal{A}}^{\ast}(\mathcal{H})$ and $ C,C'\in {GL^{+}(\mathcal{H})}$. Let $\{\Lambda_\omega\}_{w\in \Omega }$ is a continuous $(C,C')-$controlled Bessel K-g-frame for Hilbert $C^{\ast}$- module $\mathcal{H}$  with respect to $\{\mathcal{K}_{w}: w\in\Omega\}$. $\{\Lambda_\omega\}_{w\in \Omega }$ is a continuous $(C,C')-$controlled K- g-frame if and only if there exists $A>0$ such that $$AKK^\ast \leq S_{(C,C')}.$$
	
\end{lemma}
\begin{proof}
	The family $\{\Lambda_\omega\}_{w\in \Omega }$ is a continuous $(C,C')-$controlled K- g-frame if and only if 
\begin{equation} \label{do2}
   A\langle K^{\ast}f,K^{\ast}f\rangle_\mathcal{A} \leq \int_{\Omega}\langle \Lambda_\omega C f,\Lambda_\omega C'f \rangle_\mathcal{A} d\mu(w) \leq B\langle f,f\rangle_\mathcal{A} , f\in \mathcal{H}.  
\end{equation}
   If and only if, 
\begin{equation}
    \langle AKK^\ast f,f\rangle_\mathcal{A} \leq \langle S_{(C,C')}f,f\rangle_\mathcal{A}\leq \langle Bf,f\rangle_\mathcal{A} , f\in \mathcal{H}	.
\end{equation}
    If $$A\langle K^{\ast}f,K^{\ast}f\rangle_\mathcal{A} \leq \langle Sf,f\rangle_\mathcal{A} ,$$
    and the family $\{\Lambda_\omega\}_{w\in \Omega }$ is a continuous $(C,C')-$controlled Bessel K-g-frame sequence then:
    $$\langle Sf,f\rangle_\mathcal{A} \leq B\langle f,f\rangle_\mathcal{A} , f\in \mathcal{H}.$$ 
    Wich completes the proof.
    
\end{proof}
\begin{theorem}
	Let $K \in End_{\mathcal{A}}^{\ast}(\mathcal{H})$ and $ C,C'\in {GL^{+}(\mathcal{H})}$.Suppose that   $K^{\ast}$ commute with C and C'. If $\{\Lambda_\omega\}_{w\in \Omega }$ is a continuous $(C,C')-$controlled  g-frame for Hilbert $C^{\ast}$- module $\mathcal{H}$  with respect to $\{\mathcal{K}_{w}: w\in\Omega\}$, then  $\{\Lambda_\omega K^{\ast}\}_{w\in \Omega }$ is a continuous $(C,C')-$controlled K- g-frame for Hilbert $C^{\ast}$- module $\mathcal{H}$  with respect to $\{\mathcal{K}_{w}: w\in\Omega\}$.
\end{theorem}
\begin{proof}
	Let $\{\Lambda_\omega\}_{w\in \Omega }$ be a continuous $(C,C')-$controlled  g-frame for Hilbert $C^{\ast}$- module $\mathcal{H}$  with respect to $\{\mathcal{K}_{w}: w\in\Omega\}$, then, 
\begin{equation}
    A\langle f,f\rangle_\mathcal{A} \leq \int_{\Omega}\langle \Lambda_\omega C f,\Lambda_\omega C'f \rangle_\mathcal{A} d\mu(w) \leq B\langle f,f\rangle_\mathcal{A} , f\in \mathcal{H}.  
\end{equation}
    Hence, 
    $$
    A\langle K^{\ast}f,K^{\ast}f\rangle_\mathcal{A} \leq \int_{\Omega}\langle \Lambda_\omega C K^{\ast} f,\Lambda_\omega C' K^{\ast}f \rangle_\mathcal{A} d\mu(w) \leq B\langle K^{\ast}f,K^{\ast}f\rangle_\mathcal{A} , f\in \mathcal{H}.  $$
    Therefore, 
    $$
    A\langle K^{\ast}f,K^{\ast}f\rangle_\mathcal{A} \leq \int_{\Omega}\langle \Lambda_\omega  K^{\ast}C  f,\Lambda_\omega K^{\ast} C' f \rangle_\mathcal{A} d\mu(w) \leq B \|K^{\ast}\|^2\langle f,f \rangle_\mathcal{A} , f\in \mathcal{H}.  $$
    This concludes that $\{\Lambda_\omega K^{\ast}\}_{w\in \Omega }$ is a continuous $(C,C')-$controlled K- g-frame for Hilbert $C^{\ast}$- module $\mathcal{H}$  with respect to $\{\mathcal{K}_{w}: w\in\Omega\}$.	
\end{proof}
\begin{lemma}\label{do4}
	Let $K \in End_{\mathcal{A}}^{\ast}(\mathcal{H})$ and $ C,C'\in {GL^{+}(\mathcal{H})}$.Suppose that C and C' commute with each other and commute with S. Then 
	$\{\Lambda_\omega\}_{w\in \Omega }$ is a continuous $(C,C')-$controlled  K-g-frame for $\mathcal{H}$  with respect to $\{\mathcal{K}_{w}: w\in\Omega\}$ if and only if $\{\Lambda_\omega\}_{w\in \Omega }$ is a continuous $(C'C,I_{\mathcal{H}})-$controlled  K-g-frame for Hilbert $C^{\ast}$- module $\mathcal{H}$  with respect to $\{\mathcal{K}_{w}: w\in\Omega\}$.
\end{lemma}
\begin{proof}
	For all $f \in \mathcal{H} $ we have,
\begin{align*}
    \langle (C')^{-1}S_{(C,C')} C^{-1}f,f\rangle_\mathcal{A} &=\int_{\Omega}\langle C' \Lambda_\omega^{\ast} \Lambda_\omega  C  C^{-1}f, (C')^{-1} f \rangle_\mathcal{A} d\mu(w)\\ 
    &=\int_{\Omega}\langle  \Lambda_\omega^{\ast} \Lambda_\omega  f,  f \rangle_\mathcal{A} d\mu(w)\\&=\langle Sf , f\rangle_\mathcal{A},
\end{align*}
    where $$ Sf=\int_{\Omega}\Lambda_\omega^{\ast} \Lambda_\omega  f d\mu(w).$$
    Hence, $$S=(C')^{-1}S_{(C,C')} C^{-1}$$
    For any $f \in \mathcal{H}$, we have,
\begin{align*}
    \int_{\Omega}\langle  \Lambda_\omega  C f,  \Lambda_\omega C' f \rangle_\mathcal{A} d\mu(w)&=\int_{\Omega}\langle     C' \Lambda_\omega^{\ast} \Lambda_\omega C f ,f\rangle_\mathcal{A} d\mu(w)\\
    &=\langle S_{(C,C')}f , f\rangle_\mathcal{A}\\
    &=\langle C'SCf , f\rangle_\mathcal{A}\\
    &=\langle CSC'f , f\rangle_\mathcal{A}\\
    &=\langle SC'Cf , f\rangle_\mathcal{A}\\
    &=\int_{\Omega}\langle  \Lambda_\omega^{\ast} \Lambda_\omega C'C f ,f\rangle_\mathcal{A} d\mu(w)\\
    &=\int_{\Omega}\langle   \Lambda_\omega C'C f ,\Lambda_\omega f\rangle_\mathcal{A} d\mu(w)
\end{align*}
    Hence, $\{\Lambda_\omega\}_{w\in \Omega }$ is a continuous $(CC',I_{\mathcal{H}})-$controlled  K-g-frame for $\mathcal{H}$ with bounds A and B respect to $\{\mathcal{K}_{w}: w\in\Omega\}$ if and only if,
    $$A\langle K^{\ast}f,K^{\ast}f\rangle_\mathcal{A} \leq \int_{\Omega}\langle \Lambda_\omega  C' C  f,\Lambda_\omega   f \rangle_\mathcal{A} d\mu(w) \leq B \langle f,f \rangle_\mathcal{A} , f\in \mathcal{H}.$$
    The results holds.
\end{proof}
\begin{lemma}
	Let $K \in End_{\mathcal{A}}^{\ast}(\mathcal{H})$ and $ C,C'\in {GL^{+}(\mathcal{H})}$. Then $\{\Lambda_\omega\}_{w\in \Omega }$ is a continuous $(C,C')-$controlled  K-g-frame for Hilbert $C^{\ast}$-module$\mathcal{H}$ with  respect to $\{\mathcal{K}_{w}: w\in\Omega\}$ if and only if $\{\Lambda_\omega\}_{w\in \Omega }$ is a continuous $((C'C)^{\frac{1}{2}},((C'C)^{\frac{1}{2}})-$controlled  K-g-frame for Hilbert $\mathcal{H}$ with  respect to $\{\mathcal{K}_{w}: w\in\Omega\}$.
	 
\end{lemma}
\begin{proof}
	The proof is similar as proof of lemma \ref{do4}.
	
\end{proof}
\begin{lemma}
	Let $K \in End_{\mathcal{A}}^{\ast}(\mathcal{H})$ and $ C,C'\in {GL^{+}(\mathcal{H})}$. Suppose that $CK=KC$ , $C'K=KC'$ and $CS=SC$. Then  $\{\Lambda_\omega\}_{w\in \Omega }$ is a continuous $(C,C')-$controlled  K-g-frame for  $\mathcal{H}$  with respect to $\{\mathcal{K}_{w}: w\in\Omega\}$if and only if  $\{\Lambda_\omega\}_{w\in \Omega }$ is a continuous  K- g-frame for  $\mathcal{H}$  with respect to $\{\mathcal{K}_{w}: w\in\Omega\}$.
\end{lemma}
\begin{proof}
	Assume that $\{\Lambda_\omega\}_{w\in \Omega }$ is a continuous  K- g-frame for  $\mathcal{H}$  with bounds A and B with respect to $\{\mathcal{K}_{w}: w\in\Omega\}$.Then,
\begin{equation}
    A\langle K K^{\ast}f,f\rangle_\mathcal{A} \leq \langle S  f,   f \rangle_\mathcal{A}  \leq B \langle f,f \rangle_\mathcal{A} , f\in \mathcal{H}	.
\end{equation}
    Since C,C' are bounded positive operators, there exist constants m, m',M ,M'
    $(0<m,m',M,M'< \infty)$ such that, 
    $$mI_{\mathcal{H}} \leq C \leq MI_{\mathcal{H}}$$
    $$m'I_{\mathcal{H}} \leq C'\leq M'I_{\mathcal{H}}.$$
	From, 
	$$\langle CSf,f\rangle_\mathcal{A} = \langle f,SCf\rangle_\mathcal{A} = \langle f,CSf\rangle_\mathcal{A}.$$
	We have, 
	$$mAK K^{\ast} \leq SC=CS \leq MBI_{\mathcal{H}}.$$
    Then we have; 
	\begin{equation}
    mm'AK K^{\ast} \leq C'SC \leq MM'BI_{\mathcal{H}}.
	\end{equation}
	Therefore, 
\begin{equation}
    mm'A\langle K^{\ast}f, K^{\ast}f \rangle_\mathcal{A} \leq   \int_{\Omega}\langle \Lambda_\omega   C  f,\Lambda_\omega C'  f \rangle_\mathcal{A} d\mu(w) \leq MM'B \langle f,f \rangle_\mathcal{A} , f\in \mathcal{H}.
\end{equation}
    Hence,  $\{\Lambda_\omega\}_{w\in \Omega }$ is a continuous $(C,C')-$controlled  K-g-frame for  $\mathcal{H}$  with respect to $\{\mathcal{K}_{w}: w\in\Omega\}$.\\
    Conversely, assume that $\{\Lambda_\omega\}_{w\in \Omega }$ is a continuous $(C,C')-$controlled  K-g-frame for  $\mathcal{H}$  with respect to $\{\mathcal{K}_{w}: w\in\Omega\}$ with bounds A and B. On one hand we have For any $f \in \mathcal{H} $, on one hand we have,
\begin{align*}
    A\langle K^{\ast}f, K^{\ast}f \rangle_\mathcal{A} &= A\langle (CC')^{\frac{1}{2}}(CC')^{\frac{-1}{2}}K^{\ast}f, (CC')^{\frac{1}{2}}(CC')^{\frac{-1}{2}}K^{\ast}f \rangle_\mathcal{A}\\
    &=  A\langle (CC')^{\frac{1}{2}}K^{\ast}(CC')^{\frac{-1}{2}}f,(CC')^{\frac{1}{2}}K^{\ast}(CC')^{\frac{-1}{2}} f \rangle_\mathcal{A}\\
    &\leq \|(CC')^{\frac{1}{2}}\|^2 \int_{\Omega}\langle \Lambda_\omega C (CC')^{\frac{-1}{2}}  f,\Lambda_\omega C' (CC')^{\frac{-1}{2}} f \rangle_\mathcal{A} d\mu(w)\\
    &=\|(CC')^{\frac{1}{2}}\|^2 \langle S C (CC')^{\frac{-1}{2}}  f, C' (CC')^{\frac{-1}{2}} f \rangle_\mathcal{A}\\
    &=\|(CC')^{\frac{1}{2}}\|^2 \langle S C^{\frac{1}{2}} (C')^{\frac{-1}{2}}  f, (C')^{\frac{1}{2}} (C)^{\frac{-1}{2}} f \rangle_\mathcal{A}\\
    &=\|(CC')^{\frac{1}{2}}\|^2 \langle (C)^{\frac{-1}{2}} (C')^{\frac{1}{2}} S C^{\frac{1}{2}} (C')^{\frac{-1}{2}}  f,  f \rangle_\mathcal{A}\\
    &=\|(CC')^{\frac{1}{2}}\|^2 \langle S f,  f \rangle_\mathcal{A}.\\
\end{align*}
    So,
\begin{align*}
    A\|(CC')^{\frac{1}{2}}\|^{-2}\langle K^{\ast}f, K^{\ast}f   \rangle_\mathcal{A} &\leq \int_{\Omega}\langle \Lambda_\omega  f,\Lambda_\omega  f \rangle_\mathcal{A} d\mu(w)\\
    &= \langle S f,  f \rangle_\mathcal{A}.
\end{align*}
  
    On other hand, we have, 
\begin{align*}
    \int_{\Omega}\langle \Lambda_\omega  f,\Lambda_\omega  f \rangle_\mathcal{A} d\mu(w)&=\langle S f,  f \rangle_\mathcal{A}\\
    &=\langle (CC')^{\frac{-1}{2}}(CC')^{\frac{1}{2}}S f,  f \rangle_\mathcal{A}\\
    &= \langle (CC')^{\frac{1}{2}}S f,  (CC')^{\frac{-1}{2}} f \rangle_\mathcal{A}\\
    &= \langle  (CC')(CC')^{\frac{-1}{2}}S f,  (CC')^{\frac{-1}{2}} f \rangle_\mathcal{A}\\
    &=\langle C'SC (CC')^{\frac{-1}{2}} f,  (CC')^{\frac{-1}{2}} f \rangle_\mathcal{A}\\
    &\leq B \langle (CC')^{\frac{-1}{2}} f,  (CC')^{\frac{-1}{2}}f \rangle_\mathcal{A}\\
    &\leq B \|(CC')^{\frac{-1}{2}}\|^2 \langle  f,  f \rangle_\mathcal{A}.
\end{align*}
Therfore, $\{\Lambda_\omega\}_{w\in \Omega }$ is a continuous   K-g-frame for  $\mathcal{H}$  with respect to $\{\mathcal{K}_{w}: w\in\Omega\}$ with bounds $A\|(CC')^{\frac{1}{2}}\|^{-2}$ and $B\|(CC')^{\frac{-1}{2}}\|^2$.

\end{proof}
\begin{proposition}
	Let $K\in End_{\mathcal{A}}^{\ast}(\mathcal{H})$ and $ C,C'\in {GL^{+}(\mathcal{H})}$. Let $\{\Lambda_\omega\}_{w\in \Omega }$ be a continuous $(C,C')-$controlled  K-g-frame for  $\mathcal{H}$  with respect to $\{\mathcal{K}_{w}: w\in\Omega\}$. Suppose that $\overline{R(K^{\ast})}$ is orthogonally complemented. If $T\in End_{\mathcal{A}}^{\ast}(\mathcal{H})$ with $R(T)\subset R(K)$, then $\{\Lambda_\omega\}_{w\in \Omega }$  is a continuous $(C,C')-$controlled  T-g-frame for  $\mathcal{H}$  with respect to $\{\mathcal{K}_{w}: w\in\Omega\}$ .
	
\end{proposition}
\begin{proof}
	Let $\{\Lambda_\omega\}_{w\in \Omega }$ be a continuous $(C,C')-$controlled  K-g-frame for  $\mathcal{H}$  with respect to $\{\mathcal{K}_{w}: w\in\Omega\}$. Then there exists $A,B > 0$  such that, 
	$$A\langle K^{\ast}f, K^{\ast}f   \rangle _\mathcal{A}\leq \int_{\Omega}\langle \Lambda_\omega   C f,\Lambda_\omega C' f \rangle_\mathcal{A} d\mu(w)\leq B  \langle  f,  f \rangle_\mathcal{A} .$$
	From lemma \ref{do5} and $R(T)\subset R(K)$, there exist some $ m > 0 $ such that 
	$$TT^{\ast} \leq m KK^{\ast}.$$
	Hence, 
	$$\frac{A}{m}\langle T^{\ast}f, T^{\ast}f   \rangle_\mathcal{A} \leq A\langle K^{\ast}f, K^{\ast}f   \rangle _\mathcal{A} \leq \int_{\Omega}\langle \Lambda_\omega   C f,\Lambda_\omega C' f \rangle_\mathcal{A} d\mu(w)\leq B  \langle  f,  f \rangle_\mathcal{A}.$$
	So, $\{\Lambda_\omega\}_{w\in \Omega }$  is a continuous $(C,C')-$controlled  T-g-frame for  $\mathcal{H}$  with respect to $\{\mathcal{K}_{w}: w\in\Omega\}$ .
	
\end{proof}
\begin{theorem}
	Let $K_1, K_2\in End_{\mathcal{A}}^{\ast}(\mathcal{H})$ such that
	 $ R(K_1)\perp R(K_2) $. If $\{\Lambda_\omega\}_{w\in \Omega }$  is a continuous $(C,C')-$controlled  $K_1$-g-frame for  $\mathcal{H}$ as well a $ K_2$-g-frame for  $\mathcal{H}$ with respect to $\{\mathcal{K}_{w}: w\in\Omega\}$ and $\alpha $ , $\beta $ are scalers. Then  $\{\Lambda_\omega\}_{w\in \Omega }$ is a continuous $(C,C')-$controlled  $(\alpha K_1+\beta K_2)$-g-frame  and a continuous $(C,C')-$controlled  $(K_1 K_2)$-g-frame for  $\mathcal{H}$  with respect to $\{\mathcal{K}_{w}: w\in\Omega\}$ .	
\end{theorem}
\begin{proof}
	Let $\{\Lambda_\omega\}_{w\in \Omega }\subset End_{\mathcal{A}}^{\ast}(\mathcal{H},\mathcal{K}_{w})$  be a continuous $(C,C')-$controlled  $K_1$-g-frame for  $\mathcal{H}$ as well a $ K_2$-g-frame for  $\mathcal{H}$ with respect to $\{\mathcal{K}_{w}: w\in\Omega\}$.\\ Then there exist positive constants $A_1,A_2,B_1,B_2$ such that, 
	$$A_1\langle K_1^{\ast}f, K_1^{\ast}f   \rangle _\mathcal{A} \leq \int_{\Omega}\langle \Lambda_\omega   C f,\Lambda_\omega C' f \rangle_\mathcal{A} d\mu(w)\leq B_1  \langle  f,  f \rangle_\mathcal{A}.$$
	$$A_2\langle K_2^{\ast}f, K_2^{\ast}f   \rangle _\mathcal{A} \leq \int_{\Omega}\langle \Lambda_\omega   C f,\Lambda_\omega C' f \rangle_\mathcal{A} d\mu(w)\leq B_2  \langle  f,  f \rangle_\mathcal{A}.$$
	For any $f \in \mathcal{H} $, we have,

     $\langle (\alpha K_1+\beta K_2)^{\ast}f,(\alpha K_1+\beta K_2)^{\ast}f \rangle_\mathcal{A}$ =$\langle 	\overline{\alpha} K_1^{\ast} f+\overline{\beta} K_2^{\ast} f,\overline{\alpha}K_1^{\ast} f + \overline{\beta}K_2^{\ast} f\rangle_\mathcal{A}$
    $$=|\alpha|^2 \langle K_1^{\ast} f,K_1^{\ast}f \rangle_\mathcal{A} + \overline{\alpha} \beta  \langle K_1^{\ast} f,K_2^{\ast}f \rangle + \alpha \overline{\beta} \langle K_2^{\ast} f,K_1^{\ast}f \rangle + |\beta |^2 \langle K_2^{\ast} f,K_1^{\ast}f \rangle$$.

%\begin{equation}
   %  \langle (\alpha K_1+\beta K_2)^{\ast}f,(\alpha K_1+\beta)^{\ast}f \rangle_\mathcal{A} =\langle 	\overline{\alpha} K_1^{\ast} f+\overline{\beta} K_2^{\ast} f,\overline{\alpha}K_1^{\ast} f + \overline{\beta}K_2^{\ast} f\rangle_\mathcal{A}\\
   % =|\alpha|^2 \langle K_1^{\ast} f,K_1^{\ast}f \rangle_\mathcal{A} + \overline{\alpha} \beta  \langle K_1^{\ast} f,K_2^{\ast}f \rangle + \alpha \overline{\beta} \langle K_2^{\ast} f,K_1^{\ast}f \rangle + |\beta |^2 \langle K_2^{\ast} f,K_1^{\ast}f \rangle.
%\end{equation}
   
	Since $ R(K_1)\perp R(K_2) $,then, 
	$$\langle (\alpha K_1+\beta K_2)^{\ast}f,(\alpha K_1+\beta K_2)^{\ast}f \rangle_\mathcal{A}=|\alpha|^2 \langle K_1^{\ast} f,K_1^{\ast}f \rangle+|\beta |^2 \langle K_2^{\ast} f,K_1^{\ast}f \rangle_\mathcal{A}.$$ 
	Therefore, for each $f \in \mathcal{H}$, we have, 
	$$\frac{A_1A_2}{2(|\alpha|^2A_2+|\beta |^2A_1)}\langle (\alpha K_1+\beta K_2)^{\ast}f,(\alpha K_1+\beta K_2)^{\ast}f\rangle_\mathcal{A}$$
	$$=\frac{A_1A_2|\alpha|^2}{2(|\alpha|^2A_2+|\beta |^2A_1)}\langle K_1^{\ast}f,K_1^{\ast}f\rangle_\mathcal{A}+\frac{A_1A_2|\beta|^2}{2(|\alpha|^2A_2+|\beta |^2A_1)}\langle  K_2^{\ast}f, K_2^{\ast}f\rangle_\mathcal{A}$$
	$$\leq \frac{1}{2} \int_{\Omega}\langle \Lambda_\omega   C f,\Lambda_\omega C' f \rangle_\mathcal{A} d\mu(w)+ \frac{1}{2} \int_{\Omega}\langle \Lambda_\omega   C f,\Lambda_\omega C' f \rangle_\mathcal{A} d\mu(w)\leq \frac{B_1+B_2}{2}\langle  f,  f \rangle_\mathcal{A}.$$
	Therefore $\{\Lambda_\omega\}_{w\in \Omega }$ is a continuous $(C,C')-$controlled  $(\alpha K_1+\beta K_2)$-g-frame for  $\mathcal{H}$  with respect to $\{\mathcal{K}_{w}: w\in\Omega\}$.\\
	Also for every $f \in\mathcal{H} $ we have, 
\begin{align*}
    \langle (K_1K_2)^{\ast}f, (K_1K_2)^{\ast}f \rangle_\mathcal{A} &= \langle K_2^{\ast}K_1^{\ast}f, K_2^{\ast}K_1^{\ast}f \rangle_\mathcal{A}\\
    &\leq \|K_2^{\ast}\|^2\langle K_1^{\ast}f, K_1^{\ast}f \rangle_\mathcal{A}.
\end{align*}
   Since $\{\Lambda_\omega\}_{w\in \Omega }$ is a continuous $(C,C')-$controlled  $ K_1$-g-frame for  $\mathcal{H}$  with respect to $\{\mathcal{K}_{w}: w\in\Omega\}$, we have for every $f \in\mathcal{H} $,
   $$A_1\|K_2^{\ast}\|^{-2}  \langle (K_1K_2)^{\ast}f, (K_1K_2)^{\ast}f \rangle_\mathcal{A}\leq \int_{\Omega}\langle \Lambda_\omega   C f,\Lambda_\omega C' f \rangle_\mathcal{A} d\mu(w) \leq B_1 \langle  f,  f \rangle_\mathcal{A} .$$
   So,  $\{\Lambda_\omega\}_{w\in \Omega }$ is a continuous $(C,C')-$controlled  $ (K_1K_2)$-g-frame for  $\mathcal{H}$  with respect to $\{\mathcal{K}_{w}: w\in\Omega\}$.
\end{proof}
\begin{corollary}
	Let $K\in End_{\mathcal{A}}^{\ast}(\mathcal{H})$. If $\{\Lambda_\omega\}_{w\in \Omega }$ is a continuous $(C,C')-$controlled  $ K$-g-frame for  $\mathcal{H}$  with respect to $\{\mathcal{K}_{w}: w\in\Omega\}$, then for any operator $\ominus$ in the subalgebra generated by K, the family $\{\Lambda_\omega\}_{w\in \Omega }$ is a continuous $(C,C')-$controlled  $ \ominus$-g-frame for  $\mathcal{H}$ with respect to $\{\mathcal{K}_{w}: w\in\Omega\}$.
\end{corollary}

\bibliographystyle{amsplain}

\vspace{0.1in}

\end{document}